\documentclass[a4paper, 11pt, reqno]{amsart}
\usepackage{amsmath,amsfonts,amssymb,amsthm,enumerate}
\usepackage{graphicx}
\usepackage{xy} \xyoption{all}

\newtheorem{thm}{Theorem}[section]
\newtheorem{cor}[thm]{Corollary}
\newtheorem{prop}[thm]{Proposition}
\newtheorem{lem}[thm]{Lemma}

\theoremstyle{definition}
\newtheorem{defn}[thm]{Definition}
\newtheorem{exas}[thm]{Example}
\newtheorem{rem}[thm]{Remark}

\let\phi\varphi

\begin{document}
\title{On Leavitt path algebras of Hopf graphs}
\maketitle
\begin{center}
	T.\,G.~Nam\footnote{Institute of Mathematics, VAST, 18 Hoang Quoc Viet, Cau Giay, Hanoi, Vietnam. E-mail address: \texttt{tgnam@math.ac.vn}} and N.\,T.~Phuc\footnote{Faculty of Mathematics - Informatics Teacher Education, Dong Thap University, Vietnam. E-mail address: \texttt{ntphuc@dthu.edu.vn} 
		
		\ \ {\bf Acknowledgements}:   
		The authors were supported by the Vietnam National Foundation for Science and Technology Development (NAFOSTED) under Grant 101.04-2020.01. We also take an opportunity to express their deep gratitude to the anonymous referee for extremely careful reading, highly professional working with our manuscript, and quite valuable suggestions.}
\end{center}
 
\begin{abstract} In this paper, we provide the structure of   Hopf graphs associated to pairs $(G, \mathfrak{r})$ consisting of groups $G$ together with ramification datas $\mathfrak{r}$ and their Leavitt path algebras. Consequently, we characterize the Gelfand-Kirillov dimension, the stable rank, the purely infinite simplicity and the existence of a nonzero finite dimensional representation  of  the Leavitt path algebra of a Hopf graph via properties of ramification data $\mathfrak{r}$ and $G$.

\medskip


\end{abstract}

\section{Introduction}

Given a (row-finite) directed graph $E$ and a field $K$, Abrams and Aranda Pino in
\cite{ap:tlpaoag05}, and independently Ara, Moreno, and Pardo in \cite{amp:nktfga},
introduced the \emph{Leavitt path algebra} $L_K(E)$. Abrams and Aranda Pino later
extended the definition in \cite{ap:tlpaoag08} to all countable directed graphs. Goodearl in \cite{g:lpaadl09} extended the notion of Leavitt path algebras to all (possibly uncountable) directed graphs $E$. In \cite{tomf:lpawciacr}, Tomforde generalized the construction of Leavitt path algebras by replacing the field with a commutative ring. Katsov, Nam and Zumbragel in \cite{knz:solpawcicr} considered the concept of Leavitt path algebras with coefficients in a commutative semiring. Leavitt path algebras generalize the Leavitt algebras $L_K(1, n)$ of \cite{leav:tmtoar}, and also contain many other interesting classes of algebras. In addition, Leavitt path algebras are intimately related to graph $C^*$-algebras (see~\cite{r:ga}). During the past seventeen years, Leavitt path algebras have become a topic of intense investigation by mathematicians from across the mathematical spectrum. For a detailed history and overview of Leavitt path algebras we refer the reader to the survey article~\cite{a:firstde}.

Cibils and Rosso \cite{Hopf} have introduced the notion of the Hopf graph $\Gamma_{G, \mathfrak{r}}$ of a group $G$ with a ramification data which is a function from $\mathcal{C}$ to $\mathbb{N}$, denoted by $\mathfrak{r}= \sum\limits_{C\in \mathcal{C}} \mathfrak{r}_CC$, where $\mathcal{C}$ is the set of conjugacy classes of $G$. They then classified all graded Hopf algebras structures over path coalgebras using Hopf graphs. It turns out that the path coalgebra $KE$ of a graph $E$ over a field $K$ admits a structure of a Hopf algebra if and only if $E$ is a Hopf graph. We should note that Hopf graphs are similar to Cayley graphs, which have the set of vertices given by the elements of a group and arrows corresponding to multiplication by elements of a chosen system of generators.  In recent years, there have been several works around Leavitt path algebras of Cayley graphs of finite groups and in particular computing their Grothendieck group $K_0$ (see, e.g., \cite{ap:tlpaogcg, aeg:lpaocg, as:lpaocgafcg, moh:lpaowcg}) and regarding their simplicity and Invariant Basis Number property (see, e.g., \cite{moh:lpaowcg, np:tsolpaatibnp}). 

Motivated by the above results, in this article we investigate Leavitt path algebras of Hopf graphs. More namely, we describe  connected and strongly connected components of Hopf graphs $\Gamma_{G, \mathfrak{r}}$, and obtain that the Leavitt path algebra $L_K(\Gamma_{G, \mathfrak{r}})$ is isomorphic to the direct sum of 
$|G/\Delta^0_{G, \mathfrak{r}}|$-copies of the Leavitt path algebra $L_K(\Delta_{G, \mathfrak{r}})$, where $\Delta^0_{G, \mathfrak{r}}$ is the subgroup of $G$ generated by the set $\{c \in C\mid C\in \mathcal{C}, \mathfrak{r}_C >0\}$ and $\Delta_{G, \mathfrak{r}}$ is the subgraph of $\Gamma_{G, \mathfrak{r}}$ having set of vertices $\Delta^0_{G, \mathfrak{r}}$ (Theorem \ref{notconnected}). 
Consequently, we characterize the Gelfand-Kirillov dimension (Theorem \ref{GKdim}),  the purely infinite simplicity (Theorem \ref{purely}), the stable rank (Theorem \ref{stablerank}), as well as classify all finite dimensional representations (Theorem \ref{fdm}) of the Leavitt path algebra $L_K(\Gamma_{G, \mathfrak{r}})$ in terms of both ramification data $\mathfrak{r}$ and the subsemigroup $S_{G, \mathfrak{r}}$ of $G$ generated by the set $\{c \in C\mid C\in \mathcal{C}, \mathfrak{r}_C >0\}$. In particular, we extend the criteria for the simplicity and Invariant Basis Number property of Leavitt path algebras of Cayley graphs, introduced in \cite{moh:lpaowcg, np:tsolpaatibnp},
to Hopf graphs. We should mention that graph-theoretic characterizations on graphs $E$ of these properties for Leavitt path algebras $L_K(E)$ have been established in literatures (see, e.g, \cite{mm:gaatgkd}, \cite{aam:lpa}, \cite{AP:srolpa}, and \cite{ko:rolpa}, respectively); while based on these criteria, our characterizations are established completely on properties of both ramification data  $\mathfrak{r}$ and $G$.

Throughout this article, the
set of nonnegative integers is denoted by $\mathbb{N}$, the integers by $\mathbb{Z}$.


%

\section{The structure of Hopf graph Leavitt path algebras}

The main aim of this section is to provide fundamental properties of Hopf graphs (Proposition \ref{common}), the structure of Hopf graph and their Leavitt path algebras (Theorem \ref{notconnected}) via properties of ramification data.
\medskip

We begin this section by recalling some general notions of graph theory.

A (directed) graph $E = (E^0, E^1, r, s)$ (or shortly $E = (E^0, E^1)$)
consists of two disjoint sets $E^0$ and $E^1$, called \emph{vertices} and \emph{edges}
respectively, together with two maps $r, s: E^1 \longrightarrow E^0$.  The
vertices $r(e)$ and $s(e)$ are referred to as the \emph{range} and the  \emph{source} of the edge~$e$, respectively. The graph is called \emph{row-finite} if
$|s^{-1}(v)|< \infty$ for all $v\in E^0$. A graph $E$ is \emph{finite} if both sets $E^0$ and $E^1$ are finite. We say that $E$ is a \emph{trivial graph} if $E$ has only one vertex and no edges. A vertex~$v$ for which $s^{-1}(v)$ is empty is called a \emph{sink}; a vertex~$v$ for which
$r^{-1}(v)$ is empty is called a \emph{source}; a vertex~$v$ is called an \emph{isolated vertex}
if it is both a source and a sink; and a vertex~$v$ is \emph{regular} if $0 < |s^{-1}(v)| < \infty$. 

A \emph{finite path} $p = e_{1} \cdots e_{n}$ in a graph $E$ is a sequence of
edges $e_{1}, \cdots, e_{n}$ such that $r(e_{i}) = s(e_{i+1})$ for $i
= 1, \cdots, n-1$.  In this case, we say that the path~$p$ starts at
the vertex $s(p) := s(e_{1})$ and ends at the vertex $r(p) :=
r(e_{n})$, and has \emph{length} $|p| := n$. We consider the vertices in $E^0$ to be paths of length $0$. We denote by $\text{Path}(E)$ the set of all finite paths of $E$.
We denote by $p^0$
the set of its vertices, that is, $p^0 = \{s(e_i), r(e_i)\ |\ i = 1,\cdots ,n\}$.
A finite path $p$ of positive length is \textit{closed} if $s(p) = r(p)$, in which case $p$
is said to be \textit{based at the vertex} $s(p)$. The closed path $p$ is called a \emph{cycle} if $p$ does not pass through any of its vertices twice. A cycle $c$ is called a \emph{single cycle} if  $|r^{-1}(v)|=|s^{-1}(v)| =1$ for all $v\in c^0$. A graph $E$ is \emph{acyclic} if it has no cycles.
An edge~$f$ is an \emph{exit} for a path $p = e_{1} \cdots e_{n}$ if $s(f) = s(e_{i})$
but $f \ne e_{i}$ for some $1 \le i \le n$. For vertices $v, w\in E^0$, we write $v\geq w$ if there
exists a path in $E$ from $v$ to $w$, i.e., a finite path $p$ with $s(p) = v$ and $r(p) =w$. By an \emph{infinite path} in $E$ we mean a
sequence $q = e_{1} \cdots e_{i}\cdots$ for which $r(e_{i}) = s(e_{i+1})$ for all positive integer $i$. We denote by $E^{\infty}$ the set of all infinite paths of $E$, and by $E^{\geq \infty}$ the set $E^{\infty}$ together with the elements of $\text{Path}(E)$ whose range vertex is a singular vertex.

Let $E = (E^0, E^1, r, s)$ be a graph. We define the {\it extended graph} of $E$ as the new graph $\widehat{E}=(E^0, E^1 \cup (E^1)^*, \widehat{r}, \widehat{s})$, where $(E^1)^*=\{e^*\mid e\in E^1\}$, and the functions $\widehat{r}$ and $\widehat{s}$ are defined as
\[\widehat{r}(e)=r(e),~\widehat{s}(e)=s(e),~\widehat{r}(e^*)=s(e) \text{ and }\widehat{s}(e^*)=r(e)\text{ for any }e\in E^1.\]
We sometimes refer to the edges in the graph $E$ as {\it real edges}, the path in the graph $E$ as {\it real path} and the additional edges in $\widehat{E}$ (i.e. the elements of $(E^1)^*$) as {\it ghost edges}. 

We say that a graph $E$ is  \emph{connected} if given any two vertices $u, v \in E^0$ there exists
a path $h$ in $\widehat{E}$ for which $\widehat{s}(h)=u$ and $\widehat{r}(h)=v$. The \emph{connected components} of a graph $E$ are the graphs
$\{E_i\}_{i\in \Lambda}$ such that $E$ is the disjoint union $E=\sqcup_{i\in \Lambda}E_i$, where every $E_i$ is connected.

A graph $E$ is  called  \emph{strongly connected} if given any pair of vertices $(u, v) \in E^0\times E^0$ there exists
a path $h$  in $E$ such that $s(h)=u$ and $r(h)=v$. We denote by $\mathcal{S}_E$ the set of all strongly connected components of $E$. It is obvious that $\mathcal{S}_E$ is partially ordered by  the relation $\geq$ defined by: for all $\Pi_1, \Pi_2\in \mathcal{S}_E$, $\Pi_1 \geq \Pi_2$ if $v\geq w$ for some $v\in \Pi_1^0$ and $w\in \Pi_2^0$. 

For an arbitrary graph $E = (E^0,E^1,s,r)$ and an arbitrary field $K$, the \emph{Leavitt path algebra} $L_{K}(E)$ of the graph~$E$
\emph{with coefficients in} $K$ is the $K$-algebra generated
by the sets $E^0$ and $E^1 \cup (E^1)^*$, satisfying the following relations for all $v, w\in E^0$ and $e, f\in E^1$:
\begin{itemize}
		\item[(1)] $v w = \delta_{v, w} w$; 
		\item[(2)] $s(e) e = e = e r(e)$ and
		$r(e) e^* = e^* = e^*s(e)$;
		\item[(3)] $e^* f = \delta_{e, f} r(e)$;
		\item[(4)] $v= \sum_{e\in s^{-1}(v)}ee^*$ for any  regular vertex $v$;
\end{itemize}
where $\delta$ is the Kronecker delta.
\medskip


It is worth mentioning the following simple fact.

\begin{rem}[{cf. \cite[Proposition 1.4]{an:colpaofgalpa} and \cite[Proposition 1.2.14]{aam:lpa}}] \label{directsum}
Let $K$ be a field and $E$ a graph with its connected components $\{E_i\}_{i\in \Lambda}$. Then $L_K(E)\cong \bigoplus_{i\in \Lambda} L_K(E_i)$.
\end{rem}

We next recall the notion of Hopf graphs introduced by Cibils and Rosso in \cite{Hopf}. Let $G$ be an arbitrary group and $\mathcal{C}$ the set of all conjugacy classes of $G$. We call a \emph{ramification data} of $G$ is a function $\mathfrak{r}: \mathcal{C}\longrightarrow \mathbb{N}$, denoted by $\mathfrak{r}= \sum\limits_{C\in \mathcal{C}} \mathfrak{r}_CC$ (where $\mathfrak{r}_C := \mathfrak{r}(C)$).  The {\it support} of $\mathfrak{r}$ is the set

$$\rm{supp}(\mathfrak{r})=\{C\in \mathcal{C}\mid \mathfrak{r}_C>0\}.$$ We say that $\mathfrak{r}$ has \emph{finite support} if $\bigcup_{C\in \rm{supp}(\mathfrak{r})}C$ is a finite set. We denote by $S_{G, \mathfrak{r}}$ the \emph{subsemigroup} of $G$ generated by $\bigcup_{C\in \rm{supp}(\mathfrak{r})}C$. 

\begin{defn}[{\cite[Definition 3.1]{Hopf}}]\label{Hopf}
Let $G$ be an arbitrary group with a ramification data $\mathfrak{r}= \sum\limits_{C\in \mathcal{C}} \mathfrak{r}_CC$. The {\it Hopf graph associated to the pair} $(G,\mathfrak{r})$, denoted by $\Gamma_{G, \mathfrak{r}}$, has set of vertices $\Gamma_{G, \mathfrak{r}}^0=G$  and has $\mathfrak{r}_C$ edges from $x$ to $xc$	for each $x \in  G$ and $c\in C$.
\end{defn}

It is worth mentioning that in \cite[Theorem 3.3]{Hopf}  Cibils and Rosso showed that Hopf graphs are precisely the graphs such that the path algebra can be endowed with a graded Hopf algebra structure.\medskip

For clarification, we illustrate the notion of Hopf graphs by presenting the following examples.
\begin{exas}\label{2.5}
	Consider the trivial group $G=\{1_G\}$ and the ramification data $\mathfrak{r}=n[1_G]$. Then $\Gamma_{\mathbb{Z}, \mathfrak{r}}$ is the rose with $n$ petals graph:
	
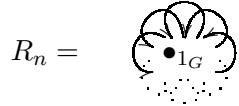
\begin{center}
		$$\xymatrix{R_n=&\bullet_{1_G} \ar@{..}@(d,r) \ar@(r,u) \ar@(u,l) \ar@(rd,ru) \ar@(lu,ld) \ar@(ur,ul) \ar@{..}@(l,d) \ar@{..}@(ld,rd)}$$ 
	\end{center}
	\medskip
	
\end{exas}
\begin{exas}\label{2.3}
Let $G = S_3$ be the symmetric group of order $6$, and write $$G=\{id, (12), (13), (23), (123), (132)\}.$$ We then have $\mathcal{C}=\{C_1:=[id], C_2:=[(12)], C_3:=[(123)]\}$, where $$C_1=\{id\}, C_2=\{(12), (13), (23)\}, C_3=\{(123), (132)\}.$$
	
(1) Consider the ramification data $\mathfrak{r}=C_3$. We then have that $\Gamma_{G, \mathfrak{r}}$ is the following graph:
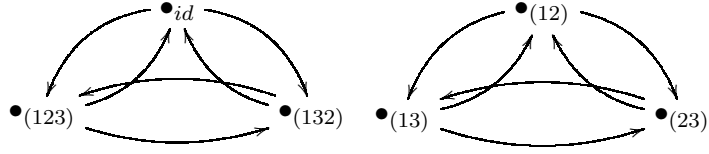
\begin{center}
	$\xymatrix{& \bullet_{id} \ar@/^1pc/[rd] \ar@/_1pc/[ld]& \\
	\bullet_{(123)} \ar@/_1pc/[ru] \ar@/_1pc/[rr]& & \bullet_{(132)} \ar@/^1pc/[lu] \ar@/_1pc/[ll]}$   
	$\xymatrix{& \bullet_{(12)} \ar@/^1pc/[rd] \ar@/_1pc/[ld]& \\
		\bullet_{(13)} \ar@/_1pc/[ru] \ar@/_1pc/[rr]& & \bullet_{(23)} \ar@/^1pc/[lu] \ar@/_1pc/[ll]}$ 
\end{center}
	\medskip
	
(2) Consider the ramification data $\mathfrak{r}=C_2$. Then $\Gamma_{G, \mathfrak{r}}$ is the following graph:
	$$
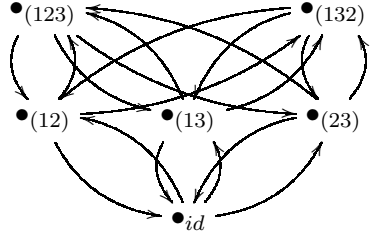
\xymatrix{\bullet_{(123)} \ar@/_1pc/[d] \ar@/_1pc/[rd] \ar@/_1pc/[rrd]& & \bullet_{(132)} \ar@/_1pc/[d] \ar@/_1pc/[ld] \ar@/_1pc/[lld]\\
	\bullet_{(12)} \ar@/_1pc/[u] \ar@/_1pc/[urr] \ar@/_1pc/[rd] & \bullet_{(13)} \ar@/_1pc/[d] \ar@/_1pc/[ul] \ar@/_1pc/[ru]& \bullet_{(23)} \ar@/_1pc/[u] \ar@/_1pc/[ull] \ar@/_1pc/[ld]\\
    & \bullet_{id} \ar@/_1pc/[u] \ar@/_1pc/[ul] \ar@/_1pc/[ru]& }$$
\end{exas}

\begin{exas}\label{2.4}
Consider the group of all integers $\mathbb{Z}$. We then have $\mathcal{C}=\{[n]\mid n\in \mathbb{Z}\}$, where $[n]=\{n\}$ for all $n\in \mathbb{Z}$.
	
	(1) Consider the ramification data $\mathfrak{r}=[0]+[2]$. Then $\Gamma_{\mathbb{Z}, \mathfrak{r}}$ is the following graph:
	
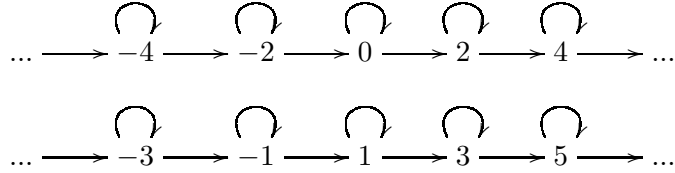
\begin{center}
	$$\xymatrix{... \ar[r]&-4 \ar@(ul,ur) \ar[r]&-2\ar@(ul,ur) \ar[r]&0\ar@(ul,ur) \ar[r]&2\ar@(ul,ur) \ar[r]&4\ar@(ul,ur) \ar[r]&...\\
	... \ar[r]&-3\ar@(ul,ur) \ar[r]&-1\ar@(ul,ur) \ar[r]&1\ar@(ul,ur) \ar[r]&3\ar@(ul,ur) \ar[r]&5\ar@(ul,ur) \ar[r]&...}$$ 
	\end{center}
	\medskip
	
	(2) Consider the ramification data $\mathfrak{r}=[2]+[3]$. Then $\Gamma_{\mathbb{Z}, \mathfrak{r}}$ is the following graph:
	$$
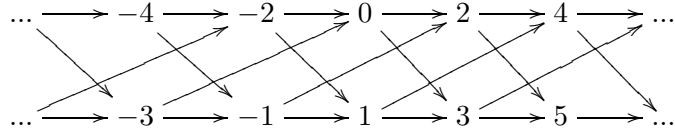
\xymatrix{... \ar[dr]\ar[r]&-4 \ar[dr] \ar[r]&-2 \ar[dr]\ar[r]&0 \ar[dr]\ar[r]&2 \ar[dr]\ar[r]&4 \ar[dr]\ar[r]&...\\
	            ... \ar[urr]\ar[r]&-3 \ar[urr]\ar[r]&-1 \ar[urr]\ar[r]&1 \ar[urr]\ar[r]&3 \ar[urr]\ar[r]&5 \ar[r]&...}$$
	\medskip
\end{exas}

The following proposition provides us with fundamental properties of Hopf graphs, which are useful to prove the main results of this article.

\begin{prop}\label{common}
Let $G$ be an arbitrary group with  a ramification data $\mathfrak{r}=\sum_{C\in \mathcal{C}}\mathfrak{r}_CC$. Then the following statements hold:
	
$(1)$ There exists a path in $\widehat{\Gamma_{G, \mathfrak{r}}}$ $(\Gamma_{G, \mathfrak{r}}$, respectively$)$ from a vertex $g$ to a vertex $h$ if and only if  $h=gw$ for some $w\in \langle S_{G, \mathfrak{r}} \rangle\ (w\in S_{G, \mathfrak{r}}$, respectively$)$, where $\langle S_{G, \mathfrak{r}} \rangle$ is the subgroup of $G$ generated by $S_{G, \mathfrak{r}}$;
	
$(2)$ $|s^{-1}(g)|=|r^{-1}(g)|$ for all $g\in G$;
	
$(3)$ $\Gamma_{G, \mathfrak{r}}$ is a row-finite graph if and only if $\mathfrak{r}$ has finite support;
	
$(4)$ If  $\mathfrak{r}= 0$, then $\Gamma_{G, \mathfrak{r}}$ is a disjoint union of isolated vertices;
	
$(5)$ If  $\mathfrak{r}\neq 0$, then $\Gamma_{G, \mathfrak{r}}$ has neither sinks nor sources;
	
$(6)$ $\Gamma_{G, \mathfrak{r}}$ has a cycle if and only if $S_{G, \mathfrak{r}}$ is a submonoid of $G$;

$(7)$ If $G$ is a finite group, then $g$ is the base of at least $\sum_{C\in \mathcal{C}}\mathfrak{r}_C|C|$ cycles for all $g\in G$.

\end{prop}
\begin{proof}
(1) Let $(\bigcup_{C\in \rm{supp}(\mathfrak{r})}C)^{-1}:=\{c^{-1}\mid c\in \bigcup_{C\in \rm{supp}(\mathfrak{r})}C\}$. There exists a path $p=e_1e_2\cdots e_n$ in $\widehat{\Gamma_{G, \mathfrak{r}}}$ from $g$ to $h$ if and only if $g=\widehat{s}(e_1)$, $h=\widehat{r}(e_n)$, and for any $1\le i\le n$, if $e_i$ is a real edge, then $\widehat{r}(e_i)=\widehat{s}(e_i)c_i$ for some $c_i\in \bigcup_{C\in \rm{supp}(\mathfrak{r})}C$, else $\widehat{r}(e_i)=\widehat{s}(e_i)c_i$ for some $c_i\in (\bigcup_{C\in \rm{supp}(\mathfrak{r})}C)^{-1}$. We then have $h=\widehat{r}(e_n)=\widehat{s}(e_1)c_1\cdots c_n=gc_1\cdots c_n$, where $c_i$'s are in $\bigcup_{C\in \rm{supp}(\mathfrak{r})}C\cup (\bigcup_{C\in \rm{supp}(\mathfrak{r})}C)^{-1}$. Let $w:= c_1\cdots c_n$. We then have $h=gw$ and $w\in \langle S_{G, \mathfrak{r}} \rangle$. If $p$ is a real path, then $c_i\in \bigcup_{C\in \rm{supp}(\mathfrak{r})}C$ for all $1\leq i \leq n$, and so $w\in S_{G, \mathfrak{r}}$, proving (1).
	
(2) It is obvious that $|s^{-1}(g)|=|r^{-1}(g)|=\sum_{C\in \mathcal{C}}\mathfrak{r}_C|C|$ for all $g\in G$.
	
(3) By item (2), $|s^{-1}(g)|=\sum_{C\in \mathcal{C}}\mathfrak{r}_C|C|$ for all $g\in G$. Since $\mathfrak{r}_C$ is a nonnegative integer for all $C\in \mathcal{C}$, we must have $|s^{-1}(g)|<\infty$ if and only if $\sum_{C\in \mathcal{C}}|C|<\infty$. Equivalently,  $\Gamma_{G, \mathfrak{r}}$ is row-finite if and only if $\mathfrak{r}$ has finite support.
	
(4) If  $\mathfrak{r}= 0$, then $|s^{-1}(g)|=|r^{-1}(g)|=0$ for all $g\in G$, and so, $\Gamma_{G, \mathfrak{r}}$ is a disjoint union of isolated vertices.
	
(5) If  $\mathfrak{r}\neq 0$, then $|s^{-1}(g)|=|r^{-1}(g)|\neq 0$ for all $g\in G$, and so, $\Gamma_{G, \mathfrak{r}}$ has neither sinks nor sources.
	
(6) $(\Longrightarrow)$. Assume that $\Gamma_{G, \mathfrak{r}}$ has a cycle $\alpha=e_1e_2\cdots e_n$. Then, there exist elements $\{c_i\}_{i=1}^{n}\subseteq \bigcup_{C\in \rm{supp}(\mathfrak{r})}C$ such that $r(e_i)=s(e_i)c_i$ for all $i=1,\cdots ,n$. Since $\alpha$ is a cycle, we have $s(e_1)=r(e_n)=s(e_1)c_1c_2\cdots c_n$, and so, $1_G=c_1\cdots c_n\in S_{G, \mathfrak{r}}$. This shows that $S_{G, \mathfrak{r}}$ is a submonoid of $G$.
	
$(\Longleftarrow)$. Assume that $S_{G, \mathfrak{r}}$ is a submonoid of $G$. This implies that there exist elements $\{c_i\}_{i=1}^{n}\subseteq \bigcup_{C\in \rm{supp}(\mathfrak{r})}C$ such that $c_1\cdots c_n=1_G$. Let $g$ be an arbitrary element of $G$. We then have that in $\Gamma_{G, \mathfrak{r}}$ there exist edges $\{e_i\}^n_{i=1}$ such that $s(e_1) = g$, $s(e_i) = gc_1\cdots c_{i-1}$ for all $i\ge 2$, and $r(e_i) = gc_1\cdots c_i$ for all $1\le i\le n$. Since $c_1\cdots c_n=1_G$, we obtain that $s(e_1) = r(e_n)$, and so, $\Gamma_{G, \mathfrak{r}}$ has a cycle $e_1 \cdots e_n$.	
	
(7) Let $g\in G$ and $c\in \bigcup_{C\in \rm{supp}(\mathfrak{r})}C$. Since $G$ is a finite group, there exists a positive integer $n$ such that $c^n=1_G$. By item (6), there exists a cycle $\alpha$ in $\Gamma_{G, \mathfrak{r}}$ based at $g$. This implies that $g$ is the base of at least $\sum_{C\in \mathcal{C}}\mathfrak{r}_C|C|$ cycles, thus finishing the proof.
%
\end{proof}

In the remainder of this section, we investigate the structure of Hopf graphs and their Leavitt path algebras. To do so, we need some useful notions and facts.

\begin{defn}
Let $G$ be a group with  a ramification data $\mathfrak{r}=\sum_{C\in \mathcal{C}}\mathfrak{r}_CC$. 

(1)	We denote by $\Delta_{G, \mathfrak{r}}$  the following subgraph of $\Gamma_{G, \mathfrak{r}}$: if $\mathfrak{r}=0$, then $\Delta_{G, \mathfrak{r}}$ is the trivial graph, else $$\Delta_{G, \mathfrak{r}}=(\Delta_{G, \mathfrak{r}}^0,\Delta_{G, \mathfrak{r}}^1,r|_{\Delta_{G, \mathfrak{r}}^1},s|_{\Delta_{G, \mathfrak{r}}^1}),$$ where $\Delta_{G, \mathfrak{r}}^0 :=\langle S_{G, \mathfrak{r}}\rangle$ and $\Delta_{G, \mathfrak{r}}^1:=r^{-1}(\Delta_{G, \mathfrak{r}}^0)$. 

(2) We denote by $\Lambda_{G, \mathfrak{r}}$  the following subgraph of $\Gamma_{G, \mathfrak{r}}$: if $S_{G, \mathfrak{r}}$ is not a submonoid of $G$, then $\Lambda_{G, \mathfrak{r}}$ is the trivial graph, else $$\Lambda_{G, \mathfrak{r}}=(\Lambda_{G, \mathfrak{r}}^0,\Lambda_{G, \mathfrak{r}}^1,r|_{\Lambda_{G, \mathfrak{r}}^1},s|_{\Lambda_{G, \mathfrak{r}}^1}),$$ where $\Lambda_{G, \mathfrak{r}}^0 :=G(S_{G, \mathfrak{r}})$ is the maximal subgroup of $G$ contained in $S_{G, \mathfrak{r}}$ and $\Lambda_{G, \mathfrak{r}}^1:=r^{-1}(\Lambda_{G, \mathfrak{r}}^0)\cap s^{-1}(\Lambda_{G, \mathfrak{r}}^0)$.
\end{defn}


For clarification, we illustrate the graphs $\Delta_{G, \mathfrak{r}}$ and  $\Lambda_{G, \mathfrak{r}}$ by presenting the following examples.
	\begin{exas}
	(1) Consider the Hopf graph $\Gamma_{G, \mathfrak{r}}$ of Example \ref{2.3} (1). We then have that $\Lambda_{G, \mathfrak{r}}=\Delta_{G, \mathfrak{r}}$ is the following graph:
	
	\begin{center}
		$
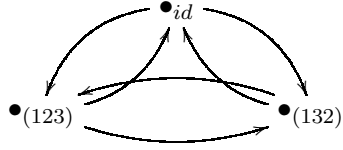
\xymatrix{& \bullet_{id} \ar@/^1pc/[rd] \ar@/_1pc/[ld]& \\
			\bullet_{(123)} \ar@/_1pc/[ru] \ar@/_1pc/[rr]& & \bullet_{(132)} \ar@/^1pc/[lu] \ar@/_1pc/[ll]}$   
	\end{center}
	\medskip
	
	(2) Consider the Hopf graph $\Gamma_{G, \mathfrak{r}}$ of Example \ref{2.3} (2). We then have $\Lambda_{G, \mathfrak{r}}=\Delta_{G, \mathfrak{r}}=\Gamma_{G, \mathfrak{r}}$.
	
	(3) Consider the Hopf graph $\Gamma_{G, \mathfrak{r}}$ of Example \ref{2.4} (1). We then have that $\Lambda_{G, \mathfrak{r}}$ is the following graph:
	
	$$
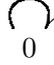
\xymatrix{0 \ar@(ul,ur)}$$
	\medskip 
	and  $\Delta_{G, \mathfrak{r}}$ is the following graph: 
	
	$$
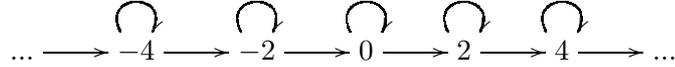
\xymatrix{... \ar[r]&-4 \ar@(ul,ur) \ar[r]&-2\ar@(ul,ur) \ar[r]&0\ar@(ul,ur) \ar[r]&2\ar@(ul,ur) \ar[r]&4\ar@(ul,ur) \ar[r]&...}$$
	\medskip
	
	(4) Consider the Hopf graph $\Gamma_{G, \mathfrak{r}}$ of Example \ref{2.4} (2). We then have that $\Lambda_{G, \mathfrak{r}}$ is the trivial graph and  $\Delta_{G, \mathfrak{r}}=\Gamma_{G, \mathfrak{r}}$.
\end{exas}

It is worth mentioning the following fact.

\begin{lem}\label{normalsubgroup}
Let $G$ be a group with  a ramification data $\mathfrak{r}=\sum_{C\in \mathcal{C}}\mathfrak{r}_CC$. Then $\Lambda_{G, \mathfrak{r}}^0$ and $\Delta_{G, \mathfrak{r}}^0$ are normal subgroups of $G$. 	
\end{lem}
\begin{proof}
We first note that for all $c\in \bigcup_{C\in \rm{supp}(\mathfrak{r})}C$ and $g\in G$, we have $gcg^{-1}\in \bigcup_{C\in \rm{supp}(\mathfrak{r})}C$ and $gc^{-1}g^{-1}\in (\bigcup_{C\in \rm{supp}(\mathfrak{r})}C)^{-1}$. We now prove that $\Lambda_{G, \mathfrak{r}}^0$ is a normal subgroup of $G$. Indeed,
let $g\in G$ and $h\in \Lambda_{G, \mathfrak{r}}^0$. Write $h=c_1c_2\cdots c_n$, where $c_i\in \bigcup_{C\in \rm{supp}(\mathfrak{r})}C$ for all $i$. We then have $$ghg^{-1}=gc_1g^{-1}gc_2g^{-1}\cdots gc_ng^{-1} \text{ and } gc_ig^{-1}\in \bigcup_{C\in \rm{supp}(\mathfrak{r})}C$$ for all $i$, and so, $ghg^{-1}\in S_{G, \mathfrak{r}}$. Moreover, since $h\in \Lambda_{G, \mathfrak{r}}^0$, $h^{-1}\in S_{G, \mathfrak{r}}$. We then have $(ghg^{-1})^{-1}=gh^{-1}g^{-1} \in S_{G, \mathfrak{r}}$. So $ghg^{-1}\in \Lambda_{G, \mathfrak{r}}^0$.
This shows that $\Lambda_{G, \mathfrak{r}}^0$ is a normal subgroup of $G$.	
	
Similarly, we obtain that $\Delta_{G, \mathfrak{r}}^0$ is a normal subgroup of $G$, thus finishing the~proof.	
\end{proof}

We are now in a position to give the main result of this section, which provides us with the structure of Hopf graphs and their Leavitt path algebras.
\begin{thm} \label{notconnected}
Let $G$ be a group with a ramification data $\mathfrak{r}=\sum_{C\in \mathcal{C}}\mathfrak{r}_CC$, and $K$ an arbitrary field. Then the following statements hold:
	
$(1)$ $\Delta_{G, \mathfrak{r}}$ is a connected component of $\Gamma_{G, \mathfrak{r}}$ and $\Gamma_{G, \mathfrak{r}}=\underset{i\in G/\Delta_{G, \mathfrak{r}}^0}{\sqcup} \Omega_i$ where $\Omega_i$ is isomorphic to $\Delta_{G, \mathfrak{r}}$ as graphs;

$(2)$ $\Lambda_{G, \mathfrak{r}}$ is a strongly connected component of $\Delta_{G, \mathfrak{r}}$ and $\Pi\cong\Lambda_{G, \mathfrak{r}}$ as graphs for all strongly connected component $\Pi$ of $\Delta_{G, \mathfrak{r}}$. Furthermore, 
\begin{align*}|\mathcal{S}_{\Delta_{G, \mathfrak{r}}}|=
	\begin{cases}
		|\Delta_{G, \mathfrak{r}}^0/\Lambda_{G, \mathfrak{r}}^0| & \textnormal{if } S_{G, \mathfrak{r}} \textnormal{ is a submonoid of $G$;}\\
		|\Delta_{G, \mathfrak{r}}^0| & \textnormal{otherwise};
	\end{cases}
\end{align*}

$(3)$ $\Delta_{G, \mathfrak{r}}$ is a strongly connected component of $\Gamma_{G, \mathfrak{r}}$ if and only if $\mathfrak{r}=0$ or $S_{G, \mathfrak{r}}$ is a subgroup of $G$. Consequently, if $S_{G, \mathfrak{r}}$ is finite, then $\Delta_{G, \mathfrak{r}}$ is a strongly connected component of $\Gamma_{G, \mathfrak{r}}$;

$(4)$ If, in addition, $G$ is commutative, then $(\mathcal{S}_{\Delta_{G, \mathfrak{r}}}, \geq )$ is downward directed; 

$(5)$ $L_K(\Gamma_{G, \mathfrak{r}})\cong L_K(\Delta_{G, \mathfrak{r}})^{(G/\Delta_{G, \mathfrak{r}}^0)}$.
\end{thm}
\begin{proof}
(1) Let $x$ and $y$ be elements of $\Delta_{G, \mathfrak{r}}^0$. Then, there exist elements $\{c_i\}^n_{i=1}\subseteq (\bigcup_{C\in \rm{supp}(\mathfrak{r})}C) \cup (\bigcup_{C\in \rm{supp}(\mathfrak{r})}C)^{-1}$  such that $x^{-1}y = c_1\cdots c_n $ and $y=x c_1\cdots c_n $. By Proposition \ref{common} (1), there exists a path in  $\widehat{\Delta_{G, \mathfrak{r}}}$ from $x$ to $y$. Let $x$ be an element of $\Delta_{G, \mathfrak{r}}^0$ and let $z$ be an element of $G$. If there exists a path in  $\widehat{\Delta_{G, \mathfrak{r}}}$ from $z$ to $x$, then by Proposition \ref{common} (1), there exists an element $w\in\langle S_{G, \mathfrak{r}} \rangle=\Delta_{G, \mathfrak{r}}^0$ such that $x=zw$, and so $z=xw^{-1}\in \Delta_{G, \mathfrak{r}}^0$. This implies that $\Delta_{G, \mathfrak{r}}$ is a connected component of $\Gamma_{G, \mathfrak{r}}$.

Let $\Omega$ be a connected component of $\Gamma_{G, \mathfrak{r}}$. We claim that $\Omega$ is isomorphic to $\Delta_{G, \mathfrak{r}}$ as graphs. Indeed, let $g$ be an arbitrary vertex of $\Omega$, and let $\phi^0: \Delta_{G, \mathfrak{r}}^0 \longrightarrow \Omega^0$ be the map defined by: $\phi^0(w)=gw$ for all $w\in \Delta_{G, \mathfrak{r}}^0$. Then, if $\phi^0(w)=gw=gw'=\phi^0(w')$, then $w=w'$, and so, $\phi^0$ is an injection. Let $h$ be an element of $\Omega^0$. Since $\Omega$ is a connected graph, there exists a path in $\widehat{\Omega}$ from $g$ to $h$. By Proposition \ref{common} (1),  $h=gw$ for some $w\in \langle S_{G, \mathfrak{r}} \rangle=\Delta_{G, \mathfrak{r}}^0$, that means, $h=\phi^0(w)$. This shows that $\phi^0$ is surjective, and hence $\phi^0$ is a bijection.

For each edge $e$ in $\Delta_{G, \mathfrak{r}}$ from $v$ to $w$, we have $w=vc$ for some $c\in \bigcup_{C\in \rm{supp}(\mathfrak{r})}C$, and $\phi^0(v)=gv$ and $\phi^0(w)=gw=gvc$. Therefore, there is a unique edge in $\Omega$ from $\phi^0(v)$ to $\phi^0(w)$ with respect to $e$. This shows that there is a bijection $\phi^1: \Delta_{G, \mathfrak{r}}^1 \longrightarrow \Omega^1$ from the set of edges of $\Delta_{G, \mathfrak{r}}$ to the set of edges of $\Omega$ such that $\phi^0$ and $\phi^1$ commute with the source and range maps, and so $\Delta_{G, \mathfrak{r}}$ is isomorphic to $\Omega$ as graphs, proving the claim.

We note that $\Omega^0= \phi^0(\Delta_{G, \mathfrak{r}}^0) = g\Delta_{G, \mathfrak{r}}^0$, and so, the cardinality of all connected components of $\Gamma_{G, \mathfrak{r}}$ is equal to the  cardinality of the quotient group $G/\Delta_{G, \mathfrak{r}}^0$, as~desired. 

(2) Let $g$ and $h$ be elements of $\Lambda_{G, \mathfrak{r}}^0$. Since $\Lambda_{G, \mathfrak{r}}^0$ is a subgroup of $G$, $g^{-1}h, h^{-1}g\in \Lambda_{G, \mathfrak{r}}^0$, and so, $g^{-1}h, h^{-1}g\in S_{G, \mathfrak{r}}$. Since $h=g(g^{-1}h), g=h(h^{-1}g)$ and by Proposition \ref{common} (1), we obtain that there exist two paths $p$ and $q$ in $\Gamma_{G, \mathfrak{r}}$ such that $s(p)=g, r(p)=h$ and $s(q)=h, r(q)=g$. Let  $g\in\Lambda_{G, \mathfrak{r}}^0$ and $h\in G$ such that there exist two paths $p$ and $q$ in $\Gamma_{G, \mathfrak{r}}$ for which $s(p)=g, r(p)=h$ and $s(q)=h, r(q)=g$. We claim that $h \in\Lambda_{G, \mathfrak{r}}^0$. Indeed, by Proposition \ref{common} (1), $h=gw$ and $g=hw'$ for some $w, w'\in S_{G, \mathfrak{r}}$. Then $h=gw=hw'w$, and so, $w'w=1_G$. It implies that $w^{-1}=w'\in S_{G, \mathfrak{r}}$, and so, $h^{-1}=w^{-1}g^{-1}\in S_{G, \mathfrak{r}}$. It is obvious that $h=gw\in S_{G, \mathfrak{r}}$, and hence, $h \in\Lambda_{G, \mathfrak{r}}^0$, proving the claim. Therefore, $\Lambda_{G, \mathfrak{r}}$ is a strongly connected component of $\Delta_{G, \mathfrak{r}}$.

We next prove that $\Pi\cong\Lambda_{G, \mathfrak{r}}$ as graphs for all strongly connected component $\Pi$ of $\Delta_{G, \mathfrak{r}}$. By repeating the approach described in the proof of (1), we consider the map $\phi^0: \Lambda_{G, \mathfrak{r}}^0 \longrightarrow \Pi^0$ defined by: $\phi^0(w)=gw$ for all $w\in \Lambda_{G, \mathfrak{r}}^0$, where $g$ is an arbitrary vertex of $\Pi$. It is sufficient to show that $\phi^0$ is surjective. Let $h$ be an element of $\Pi^0$. Since $\Pi$ is a strongly connected graph, there exist two paths in $\Gamma_{G, \mathfrak{r}}$ from $g$ to $h$ and from $h$ to $g$. By using the above argument again, we have $h=gw$ and $g=hw^{-1}$ for some $w, w^{-1}\in S_{G, \mathfrak{r}}$. It implies that $w\in \Lambda_{G, \mathfrak{r}}^0$, and so, $h=gw=\phi^0(w)$, as desired.

If $S_{G, \mathfrak{r}}$ is not a submonoid of $G$, then $\Lambda_{G, \mathfrak{r}}$ is the trivial graph, and so, $|\mathcal{S}_{\Delta_{G, \mathfrak{r}}}|=|\Delta_{G, \mathfrak{r}}^0|$. Otherwise, if $S_{G, \mathfrak{r}}$ is a submonoid of $G$, then, similarly to the proof of (1), we obtain that the cardinality of all strongly connected components of $\Delta_{G, \mathfrak{r}}$ is equal to the  cardinality of the quotient group $\Delta_{G, \mathfrak{r}}^0/\Lambda_{G, \mathfrak{r}}^0$, thus showing (2).

(3) $(\Longleftarrow)$ If $\mathfrak{r}=0$, then $\Delta_{G, \mathfrak{r}}$ is a trivial graph, and so, it is a strongly connected component of $\Gamma_{G, \mathfrak{r}}$. Suppose that $\mathfrak{r}\neq 0$ and $S_{G, \mathfrak{r}}$ is a subgroup of $G$. We then have $\Delta_{G, \mathfrak{r}}^0=\Lambda_{G, \mathfrak{r}}^0=S_{G, \mathfrak{r}}$. Therefore, $\Delta_{G, \mathfrak{r}}=\Lambda_{G, \mathfrak{r}}$ is a strongly connected component of $\Gamma_{G, \mathfrak{r}}$ by item (2).

$(\Longrightarrow)$ Suppose that $\Delta_{G, \mathfrak{r}}$ is a strongly connected component of $\Gamma_{G, \mathfrak{r}}$. If $\Delta_{G, \mathfrak{r}}$ is the trivial graph, then $\mathfrak{r}=0$. Assume that $\Delta_{G, \mathfrak{r}}$ has an edge. We claim that $S_{G, \mathfrak{r}}$ is a subgroup of $G$. Indeed, let $c\in \bigcup_{C\in \rm{supp}(\mathfrak{r})}C$ and $e$ an edge in $\Delta_{G, \mathfrak{r}}$ such that $r(e)=s(e)c$.
Since $\Delta_{G, \mathfrak{r}}$ is a strongly connected component of $\Gamma_{G, \mathfrak{r}}$, there exists a path $p_e$ in $\Delta_{G, \mathfrak{r}}$ such that $s(p_e)=r(e)$ and $r(p_e)=s(e)$, and so, $ep_e$ is a cycle in  $\Delta_{G, \mathfrak{r}}$. By Proposition \ref{common} (6), $S_{G, \mathfrak{r}}$ is a submonoid of $G$. We write $p_e=f_1\cdots f_n$ where $r(f_i)=s(f_i)w_i$ for some $w_i\in \bigcup_{C\in \rm{supp}(\mathfrak{r})}C$,  $1\leq i \leq n$. Since $ep_e$ is a cycle in  $\Delta_{G, \mathfrak{r}}$, $s(e)=r(p_e)=r(f_n)=s(f_n)w_n=\cdots =s(e)cw_1\cdots w_n$, and so, $cw_1\cdots w_n=1_G$. It shows that $S_{G, \mathfrak{r}}$ is a subgroup of $G$, proving the claim.

Assume that $S_{G, \mathfrak{r}}$ is a finite subsemigroup of $G$. If $S_{G, \mathfrak{r}}=\varnothing$, then $\Delta_{G, \mathfrak{r}}$ is the trivial graph, therefore, $\Delta_{G, \mathfrak{r}}$ is a strongly connected component of $\Gamma_{G, \mathfrak{r}}$. Otherwise, $S_{G, \mathfrak{r}}$ is a nonempty finite subsemigroup of $G$, then $S_{G, \mathfrak{r}}$ is also a subgroup of $G$, and so, $\Delta_{G, \mathfrak{r}}$ is a strongly connected component of $\Gamma_{G, \mathfrak{r}}$.

(4) Assume that $G$ is a commutative group.	Let $x, y\in \Delta_{G, \mathfrak{r}}^0$. We claim that there exists an element $z\in \Delta_{G, \mathfrak{r}}^0$ such that $x\geq z$ and $y\geq z$. Since $G$ is a commutative, we can write $x=c_1c_2\cdots c_nd_1^{-1}d_2^{-1}\cdots d_m^{-1}$ and $y=k_1k_2\cdots k_{n'}l_1^{-1}l_2^{-1}\cdots l_{m'}^{-1}$, where $c_i, d_j, k_{i'}, l_{j'}\in \bigcup_{C\in \rm{supp}(\mathfrak{r})}C$ for all $1\leq i\leq n$, $1\leq j\leq m$, $1\leq i'\leq n'$, $1\leq j'\leq m'$. Consider the path $p=e_1\cdots e_m$ and $q=f_1\cdots f_{n'}$ such that $s(e_1)=x, r(e_i)=s(e_i)d_i$, $1\leq i \leq m$ and $s(f_1)=r(p), r(f_i)=s(f_i)k_i$, $1\leq i \leq n'$. We then have $s(pq)=x$ and $r(pq)=c_1c_2\cdots c_nk_1k_2\cdots k_{n'}=:z\in \Delta_{G, \mathfrak{r}}^0$, and so, $x\geq z$. Similarly, there exist  paths $p'$, $q'$ in $\Delta_{G, \mathfrak{r}}$ such that $s(p'q')=y$, $r(p'q')=k_1k_2\cdots k_{n'}c_1c_2\cdots c_n=z$, and so, $y\geq z$, proving the claim. This implies that $(\mathcal{S}_{\Delta_{G, \mathfrak{r}}},\geq)$ is downward directed. 
	
(5) It follows from item (1) and Remark \ref{directsum}, thus finishing the proof.
\end{proof}

It is worth mentioning the following note.

\begin{rem} There exist two graphs $E$ and $F$ for which $E$ is a Hopf graph and $F$ is not a Hopf graph, but $L_K(E)\cong L_K(F)$, where $K$ is an arbitrary field. For example, let $E=\Gamma_{G, \mathfrak{r}}$ be the Hopf graph of Example \ref{2.5} with $n=3$ and $F$ the following graph: 
	$$
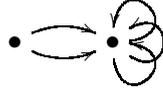
\xymatrix{\bullet \ar@/^.5pc/[r] \ar@/_.5pc/[r] &\bullet \ar@(r,u) \ar@(d,r) \ar@(rd,ru)}$$\\
Then, since $F$ has two strongly connected components which are not isomorphic to each other, and by Theorem \ref{notconnected} (2), $F$ is not a Hopf graph. It is well known that $L_K(E)\cong L_K(1,3)$ and $L_K(F)\cong M_3(L_K(1,3))$ (by  \cite[Theorem 2.11 (2)]{np:tsolpaatibnp}). By \cite[Theorem 4.4]{aap:ibLpatmr}, we have $L_K(1,3)\cong M_3(L_K(1,3))$, and so $L_K(E)\cong L_K(F)$.
\end{rem}

We close this section with the following useful corollary.

\begin{cor}\label{singlecycle}
Let $G$ be a group with  a ramification data $\mathfrak{r}=\sum_{C\in \mathcal{C}}\mathfrak{r}_CC$. Then the following statements hold:
	
$(1)$ If $\mathfrak{r}=0$, then $\Gamma_{G, \mathfrak{r}}$ is a disjoint union of isolated vertices  and $$L_K(\Gamma_{G, \mathfrak{r}}) = K^{(G)}.$$
	
$(2)$ If $\sum_{C\in \mathcal{C}}\mathfrak{r}_C|C|= 1$ and $S_{G, \mathfrak{r}}$ is a submonoid of $G$, then $\Gamma_{G, \mathfrak{r}}$ is a disjoint union of single cycles, $S_{G, \mathfrak{r}}$ is a finite normal subgroup of $G$ and $$L_K(\Gamma_{G, \mathfrak{r}}) =M_{|S_{G, \mathfrak{r}}|}(K[x, x^{-1}])^{(G/S_{G, \mathfrak{r}})}.$$
\end{cor}
\begin{proof}
(1) It immediately follows from Proposition \ref{common} (4) and Theorem \ref{notconnected}. 
	
(2) Assume that $\sum_{C\in \mathcal{C}}\mathfrak{r}_C|C|= 1$ and $S_{G, \mathfrak{r}}$ is a submonoid of $G$. We then have $\bigcup_{C\in \rm{supp}(\mathfrak{r})}C=\{c\}$ for some $c\in G$. Since $S_{G, \mathfrak{r}}$ is a submonoid of $G$, there exists a positive integer $n$ such that $c^n=1_G$, and so, $S_{G, \mathfrak{r}}=\Delta_{G, \mathfrak{r}}^0$ is the cyclic subgroup of $G$ generated by $c$, and $S_{G, \mathfrak{r}}$ is also a normal subgroup of $G$ (by Lemma \ref{normalsubgroup}).  By \cite[Theorem 2.11]{np:tsolpaatibnp}, $L_K(\Delta_{G, \mathfrak{r}})=M_{{|S_{G, \mathfrak{r}}|}}(K[x, x^{-1}])$. Then, by Theorem \ref{notconnected},  $\Gamma_{G, \mathfrak{r}}$ is a disjoint union of single cycles and $L_K(\Gamma_{G, \mathfrak{r}}) = M_{|S_{G, \mathfrak{r}}|}(K[x, x^{-1}])^{(G/S_{G, \mathfrak{r}})}$, thus finishing the proof. 
\end{proof}

\section{Applications}
In this section, based on Theorem \ref{notconnected}, we characterize the Gelfand-Kirillov dimension (Theorem \ref{GKdim}), the stable rank (Theorem \ref{stablerank}), the purely infinite simplicity (Theorem \ref{purely}) and the existence of a nonzero finite dimensional representation (Theorem \ref{fdm}) of  Leavitt path algebras of Hopf graphs via ramification datas.

\subsection{Gelfand-Kirillov dimension}
We begin this subsection by recalling some general notions on the Gelfand-Kirillov dimension of algebras. Given a field $K$ and a finitely generated $K$-algebra $A$. The \textit{Gelfand-Kirillov dimension} of $A$ ($\text{GKdim(A)}$ for short) is defined to be
$$\text{GKdim}(A) :=\limsup\limits_{n\rightarrow \infty}\log_{n}(\dim(V^{n})),$$
where $V$ is a finite dimensional subspace of $A$ that generates $A$ as an algebra over $K$. This definition is independent of the choice of $V$. If $A$ does not happen to be finitely generated over $K$, the Gelfand-Kirillov dimension of $A$ is defined to be
$$\text{GKdim}(A) = \sup\{\text{GKdim}(B)\mid B \text{ is a finitely generated subalgebra of } A\}.$$ 

In \cite{aajz:lpaofgkd} Alahmadi, Alsulami, Jain and Zelmanov  determined the Gelfand-Kirillov dimension of Leavitt path algebras of finite graphs. In \cite{mm:gaatgkd} Moreno-Fern\'andez and Siles Molina extended this to arbitrary graphs. We should mention this result here. To do so, we need to recall useful notions of graph theory.

Let $E$ be an arbitrary graph. A cycle $c$ in $E$ is said to be an {\it exclusive cycle} if it is disjoint with every other cycle; equivalently, no vertex on $c$ is the base of a different cycle other than  a cyclic permutation of $c$. We say that $E$ satisfies {\it Condition} (EXC) if every cycle of $E$ is an exclusive cycle.

For two cycles $c$ and $c'$, we write $c\Rightarrow c'$ if there exists a path that starts in $c$ and ends in $c'$. A sequence of cycles $c_1, \cdots , c_k$ is a chain of length $k$ if $c_1\Rightarrow \cdots \Rightarrow c_k$. We say that such a chain has an {\it exit} if the cycle $c_k$ has an exit. Let $d_1$ be the maximal length of a chain of cycles in $E$, and let $d_2$ be the maximal length of a chain of cycles with an exit in $E$. For every field $K$, by \cite[Theorem 3.21]{mm:gaatgkd}, 
$\text{GKdim}(L_K(E))$ is finite if and only if $E$ satisfies Condition (EXC) and the maximal length of chains of cycles in $E$ is finite. In this case, $\text{GKdim}(L_K(E)) = \max\{2d_1-1, 2d_2\}$. Consequently, every positive integer can arise as the Gelfand-Kirillov dimension of a Leavitt path algebra.

\begin{lem}\label{sequencecycle}
Let $G$ be a group with  a ramification data $\mathfrak{r}=\sum_{C\in \mathcal{C}}\mathfrak{r}_CC$ such that $\sum_{C\in \mathcal{C}}\mathfrak{r}_C|C|\geq 2$. Then,  for every cycle $\alpha$ in  $\Gamma_{G, \mathfrak{r}}$,  there exists a cycle $\beta$ in  $\Gamma_{G, \mathfrak{r}}$ such that $\beta$ is not a cyclic permutation of $\alpha$ and $\beta\Rightarrow \alpha$. 

\end{lem}
\begin{proof}
Let $\alpha= e_1\cdots e_n$ be a cycle in $\Gamma_{G, \mathfrak{r}}$. We claim that $\beta\Rightarrow \alpha$ for some cycle $\beta$ in $\Gamma_{G, \mathfrak{r}}$. Indeed, by Proposition \ref{common} (6), $S_{G, \mathfrak{r}}$ is a submonoid of $G$, and so, there exists elements $\{c_i\}_{i=1}^{n}\subseteq \bigcup_{C\in \rm{supp}(\mathfrak{r})}C$ such that $c_1\cdots c_n=1_G$, $s(\alpha)=s(e_1)$ and $s(e_i) = s(\alpha)c_1\cdots c_{i-1}$ for all $2\le i\le n$, and $r(e_i) = s(\alpha)c_1\cdots c_i$ for all $1\le i \le n$. We consider the following two cases:
	
\emph{Case }1: $|\bigcup_{C\in \rm{supp}(\mathfrak{r})}C|=1$. We then have  $\rm{supp}(\mathfrak{r})=\{C\}$ and $C=\{c\}$ for some $c\in G$, and  $\mathfrak{r}_C\geq 2$. Therefore, by the definition of $\Gamma_{G, \mathfrak{r}}$, there exists $\mathfrak{r}_C$ edges from  $s(e_i)$ to $r(e_i)$ for all $1\le i \le n$, and so, $s(\alpha)$ is base of at least $\mathfrak{r}_C$ cycles, that means, there exists a cycle $\beta$ in  $\Gamma_{G, \mathfrak{r}}$ such that $\beta$ is not a cyclic permutation of $\alpha$ and $\beta\Rightarrow \alpha$, as desired.
	
\emph{Case }2: $|\bigcup_{C\in \rm{supp}(\mathfrak{r})}C| \ge 2$. 
We consider the following subcases:

\emph{Case 2.1}: $x^m= 1_G$ for some $x\in \bigcup_{C\in \rm{supp}(\mathfrak{r})}C\setminus\{c_n\}$ and positive integer $m$. Then, $s(\alpha)$ is the base of a cycle $\beta = f_1\cdots f_{|x|}$ which is different from $\alpha$, where $|x|$ is the order of $x$, $s(f_1) = s(\alpha)$, $s(f_i) = s(\alpha)x^{i-1}$ for all $2\le i\le n$, and $r(f_i) = s(\alpha)x^i$ for all $1\le i\le n$. This implies that $\beta\Rightarrow \alpha$, as desired.

\emph{Case 2.2}: $x^m \neq 1_G$ for all $x\in \bigcup_{C\in \rm{supp}(\mathfrak{r})}C\setminus\{c_n\}$ and positive integer $m$.
Let $x$ be an arbitrary element of $\bigcup_{C\in \rm{supp}(\mathfrak{r})}C\setminus\{c_n\}$. 
 We then have $s(\alpha)x^{-1} \neq s(\alpha)$, $s(\alpha)x^{-1} \neq s(\alpha)c_1\cdots c_{n-1}$ and 
there exists an edge in $\Gamma_{G, \mathfrak{r}}$ from $s(\alpha)x^{-1}$ to $s(\alpha)$. Moreover, we receive that $s(\alpha)x^{-1}$ is the base of the cycle $\beta = f_1\cdots f_n $, where $s(\alpha)x^{-1}=s(f_1)$ and $s(f_i) = s(\alpha)x^{-1}c_1\cdots c_{i-1}$ for all $2\le i\le n$, and $r(f_i) = s(\alpha)x^{-1}c_1\cdots c_i$ for all $1\le i \le n$. These observations show that there exists a cycle $\beta$ in  $\Gamma_{G, \mathfrak{r}}$ such that $\beta$ is not a cyclic permutation of $\alpha$ and $\beta\Rightarrow \alpha$, as desired.

Therefore, in any case, we arrive at that  there exists a cycle $\beta$ in  $\Gamma_{G, \mathfrak{r}}$ such that $\beta$ is not a cyclic permutation $\alpha$ and $\beta\Rightarrow \alpha$, thus finishing the proof.
\end{proof}

In the following theorem, we compute the Gelfand-Kirillov dimension of Leavitt path algebras of Hopf graphs via ramification datas, showing that for Hopf graphs, there are only very few possible Gelfand-Kirillov dimensions.

\begin{thm}\label{GKdim}
	Let $G$ be a group with  a ramification data $\mathfrak{r}=\sum_{C\in \mathcal{C}}\mathfrak{r}_CC$ and $K$ an arbitrary field. Then 
	\begin{align*}\rm{GKdim}(L_K(\Gamma_{G, \mathfrak{r}}))=
		\begin{cases}
			\infty & \textnormal{if } \sum_{C\in \mathcal{C}}\mathfrak{r}_C|C|\geq 2 \textnormal{ and } S_{G, \mathfrak{r}} \textnormal{ is a submonoid of $G$;}\\
			1      & \textnormal{if } \sum_{C\in \mathcal{C}}\mathfrak{r}_C|C|=1 \textnormal{ and } S_{G, \mathfrak{r}} \textnormal{ is a submonoid of $G$;} \\
			0      & \textnormal{otherwise}.
		\end{cases}
	\end{align*}
\end{thm}
\begin{proof} 	
If  $\sum_{C\in \mathcal{C}}\mathfrak{r}_C|C|=1$ and $S_{G, \mathfrak{r}}$ is a submonoid of $G$, then by Corollary \ref{singlecycle} (2),  $\Gamma_{G, \mathfrak{r}}$ is a disjoint union of single cycles, and so, $\text{GKdim}(L_K(\Gamma_{G, \mathfrak{r}}))=1$, by \cite[Theorem 3.21]{mm:gaatgkd}. If $\sum_{C\in \mathcal{C}}\mathfrak{r}_C|C|\geq 2$ and $S_{G, \mathfrak{r}}$ is a submonoid of $G$, then by Proposition \ref{common} (6), $\Gamma_{G, \mathfrak{r}}$ has a cycle. Then, by Lemma \ref{sequencecycle}, we obtain that either $\Gamma_{G, \mathfrak{r}}$ does not satisfy Condition (EXC) or  $\Gamma_{G, \mathfrak{r}}$ satisfies Condition (EXC) and the maximal length of chains of cycles in $\Gamma_{G, \mathfrak{r}}$ is infinite. By \cite[Theorem 3.21]{mm:gaatgkd} again, we immediately get that $\text{GKdim}(L_K(\Gamma_{G, \mathfrak{r}}))=\infty$. Otherwise, $S_{G, \mathfrak{r}}$ is not a submonoid of $G$, and so, $\Gamma_{G, \mathfrak{r}}$ has no cycles, by Proposition \ref{common} (6). Then, by \cite[Theorem 3.21]{mm:gaatgkd},  $\text{GKdim}(L_K(\Gamma_{G, \mathfrak{r}}))=0$, thus finishing the proof.
\end{proof}

For clarification, we illustrate the statement of Theorem \ref{GKdim} by presenting the following examples.
\begin{exas}\label{exGK}
(1)  Consider the Hopf graph $\Gamma_{G, \mathfrak{r}}$ of Example \ref{2.3} (1). We then have that $\sum_{C\in \mathcal{C}}\mathfrak{r}_C|C|=2$ and $S_{G, \mathfrak{r}}=\{id, (123), (132)\}$ is a submonoid of $G=S_3$, and so $\text{GKdim}(L_K(\Gamma_{G, \mathfrak{r}}))=\infty$, by Theorem \ref{GKdim}. 

(2) Consider the Hopf graph $\Gamma_{G, \mathfrak{r}}$ of Example \ref{2.4} (1). We then have that $\sum_{C\in \mathcal{C}}\mathfrak{r}_C|C|=2$ and $S_{G, \mathfrak{r}}=\{0, 2n\mid n\geq 1\}$ is a submonoid of $G=\mathbb{Z}$, and so  $\text{GKdim}(L_K(\Gamma_{G, \mathfrak{r}}))=\infty$, by Theorem \ref{GKdim}.

(3) Consider the Hopf graph $\Gamma_{G, \mathfrak{r}}$ of Example \ref{2.5} with $n=1$. We have that $\sum_{C\in \mathcal{C}}\mathfrak{r}_C|C|=n=1$ and $S_{G, \mathfrak{r}}=\{1_G\}$ is a submonoid of $G$, and so $\text{GKdim}(L_K(\Gamma_{G, \mathfrak{r}}))=1$, by Theorem \ref{GKdim}.
		
(4) Consider the Hopf graph $\Gamma_{G, \mathfrak{r}}$ of Example \ref{2.4} (2). We have that $S_{G, \mathfrak{r}}=\{2m+3n \mid m,n\geq1\}$ is not a submonoid of $G=\mathbb{Z}$, and so $\text{GKdim}(L_K(\Gamma_{G, \mathfrak{r}}))=0$, by Theorem \ref{GKdim}.		
\end{exas}

\subsection{Purely infinite simplicity} 
An idempotent $e$ in a ring $R$ is called \emph{infinite} if $eR$ is isomorphic as a right $R$-module to a proper direct summand of itself. $R$ is called \emph{purely infinite} in case every right ideal of $R$ contains an infinite idempotent. 
In \cite{ap:pislpa06, ap:tlpaoag08} Abrams and Aranda Pino  provided criteria for Leavitt path algebras of countable graphs to be purely infinite simple. This result was extended to arbitrary graphs in \cite[Theorem 3.1.10]{aam:lpa}. We should mention the result in Theorem \ref{pis-lpas} below. To do so, we need to recall some notions.

Let $E$ be a graph and $H$ a subset of $E^0$. We say $H$ is \emph{hereditary} if for all $v\in H$ and $w\in E^0$, $v\geq w$ implies $w\in H$. We say $H$ is \emph{saturated} if whenever $v\in E^0$ has the property that $v$ is regular, $\{r(e), e\in s^{-1}(v)\}\subseteq H$, then $v\in H$. 
We say that a vertex $v \in E^0$ is \emph{cofinal} if for every $\gamma \in E^{\geq \infty}$ there is a vertex $w$ in the path $\gamma$ such that $v\geq w$. We say that a graph $E$ is \emph{cofinal} if every vertex in $E$ is cofinal.
We should note that a graph $E$ is cofinal if and only if the only hereditary and saturated subset of $E^0$ are $\varnothing$ and $E^0$ (see, e.g. \cite[Lemma 2.9.6]{aam:lpa}). 

The following theorem provides us with criteria for Leavitt path algebras of graphs to be purely infinite simple, which is very useful to prove the main result of this subsection.

\begin{thm}[{\cite[Theorem 3.1.10]{aam:lpa}}]\label{pis-lpas}
The Leavitt path algebra $L_K(E)$ of a graph $E$ with coefficients in a field $K$ is purely infinite simple if and only if the following conditions are satisfied:	

$(1)$ The only hereditary and saturated subsets of $E^0$ are $\varnothing$ and $E^0$;

$(2)$ Every cycle in $E$ has an exit;

$(3)$ $E$ has a cycle. 
\end{thm}


We are now in a position to give the main result of this subsection, which provides criteria for Leavitt path algebras of Hopf graphs to be purely infinite simple via ramification datas which plays an important role in the proof of Theorem \ref{stablerank} below. 

\begin{thm}\label{purely}
Let $G$ be a group with a ramification data $\mathfrak{r}=\sum_{C\in \mathcal{C}}\mathfrak{r}_CC$ and $K$ an arbitrary field. Then $L_K(\Gamma_{G, \mathfrak{r}})$ is purely infinite simple if and only if $S_{G, \mathfrak{r}} =G$ and
$\sum_{C\in \mathcal{C}}\mathfrak{r}_C|C|\geq 2$.
\end{thm}
\begin{proof}
$(\Longrightarrow)$. Assume that $L_K(\Gamma_{G, \mathfrak{r}})$ is purely infinite simple.  By Theorem \ref{pis-lpas}, $L_K(\Gamma_{G, \mathfrak{r}})$ has a cycle, and so, $S_{G, \mathfrak{r}}$ is a submonoid of $G$, by Proposition \ref{common} (6). 
If $\mathfrak{r}=0$, then by Proposition \ref{common} (4),  $\Gamma_{G, \mathfrak{r}}$ is a disjoint union of isolated vertices, a contradiction. Consider the case $\sum_{C\in \mathcal{C}}\mathfrak{r}_C|C|=1$. Since $S_{G, \mathfrak{r}}$ is a submonoid of $G$ and by Corollary \ref{singlecycle} (2), $L_K(\Gamma_{G, \mathfrak{r}}) =M_{|S_{G, \mathfrak{r}}|}(K[x, x^{-1}])^{(G/\Delta_{G, \mathfrak{r}}^0)}$, and so, $L_K(\Gamma_{G, \mathfrak{r}})$ is not simple, a contradiction. Therefore, we have $\sum_{C\in \mathcal{C}}\mathfrak{r}_C|C|\geq 2.$

If $\langle S_{G, \mathfrak{r}} \rangle \neq G$, then by Theorem \ref{notconnected}, we have $L_K(\Gamma_{G, \mathfrak{r}})= L_K(\Gamma_{G, \mathfrak{r}})^{(G/\langle S_{G, \mathfrak{r}} \rangle)}$, and so, $L_K(\Gamma_{G, \mathfrak{r}})$ is not simple, a contradiction. This implies that $\langle S_{G, \mathfrak{r}} \rangle=G$. 	

We claim that $S_{G, \mathfrak{r}}$ is a group. Indeed, if $S_{G, \mathfrak{r}} = \{1_G\}$, then the claim is obvious. Consider the case when $S_{G, \mathfrak{r}} \neq \{1_G\}$.
Let $d\in \bigcup_{C\in \rm{supp}(\mathfrak{r})}C$ and $d\neq 1_G$. Then, there exists an edge $e$ in $\Gamma_{G, \mathfrak{r}}$ such that $s(e)=d$ and $r(e)=d^2$. Since $S_{G, \mathfrak{r}}$ is a submonoid of $G$, there exists elements $\{w_i\}_{i=1}^{n}\subseteq \bigcup_{C\in \rm{supp}(\mathfrak{r})}C$ such that $w_1\cdots w_n=1_G$, and so, there exists a cycle $\alpha=e_1e_2\cdots e_n$ in $\Gamma_{G, \mathfrak{r}}$ such that $d=s(e_1)$ and $r(e_i)=s(e_i)w_i$ for all $1\le i\le n$. Since $L_K(\Gamma_{G, \mathfrak{r}})$ is simple and by \cite[Theorem 2.9.7]{aam:lpa}, $\Gamma_{G, \mathfrak{r}}$ is cofinal. Then, there exists a path $p=f_1\cdots f_k$ in $\Gamma_{G, \mathfrak{r}}$ which starts at $d^2$ and ends in $\alpha$. Write $r(f_i)=s(f_i)c_i$ for all $1\le i\le k$, where $c_i\in \bigcup_{C\in \rm{supp}(\mathfrak{r})}C$. Since $p$ ends in $\alpha$, $r(p)=s(e_i)$ for some $1\le i\le n$. Let $\beta :=ef_1\cdots f_ke_ie_{i+1}\cdots e_n$. We then have that $\beta$ is a closed path in $\Gamma_{G, \mathfrak{r}}$ based at $d^2$, and so, $s(e)=d=r(e_n)=s(e)dc_1\cdots c_kw_i\cdots w_n$ and $dc_1\cdots c_kw_i\cdots w_n=1_G$. This implies that $d$ is invertible in $S_{G, \mathfrak{r}}$, and hence $S_{G, \mathfrak{r}}$ is a subgroup of $G$, proving the claim. Then, since 
$\langle S_{G, \mathfrak{r}} \rangle=G$, we obtain that $ S_{G, \mathfrak{r}} =G$.

$(\Longleftarrow)$. Since $S_{G, \mathfrak{r}} =G$ and by Theorem \ref{notconnected} and Proposition \ref{common} (6), $\Gamma_{G, \mathfrak{r}}$ is both strongly connected and has a cycle. It implies that $\Gamma_{G, \mathfrak{r}}$ is cofinal, and so, the only hereditary and saturated subset of $\Gamma^0_{G, \mathfrak{r}}$ are $\varnothing$ and $\Gamma^0_{G, \mathfrak{r}}$.  
 Since $\sum_{C\in \mathcal{C}}\mathfrak{r}_C|C|\geq 2$, every cycle in $\Gamma_{G, \mathfrak{r}}$ has an exit. Then, by Theorem \ref{pis-lpas}, $L_K(\Gamma_{G, \mathfrak{r}})$ is purely infinite simple, thus finishing the proof.
\end{proof}

Consequently, we obtain the following corollary which extends \cite[Proposition 4.1]{np:tsolpaatibnp} and \cite[Theorem 3.1]{moh:lpaowcg} to Hopf graphs.

\begin{cor}\label{purelyfinite}
Let $G$ be a finite group with  a ramification data $\mathfrak{r}=\sum_{C\in \mathcal{C}}\mathfrak{r}_CC\neq 0$ and $K$ an arbitrary field. Then the following statements are equivalent:
	
$(1)$ $L_K(\Gamma_{G, \mathfrak{r}})$ is purely infinite simple;
	
$(2)$ $L_K(\Gamma_{G, \mathfrak{r}})$ is simple;
	
$(3)$ $S_{G, \mathfrak{r}}=G$ and $\sum_{C\in \mathcal{C}}\mathfrak{r}_C|C|\geq 2$.
\end{cor}
\begin{proof} The equivalence of (1) and (3) immediately follows from Theorem \ref{purely}. We note that by Theorem \ref{pis-lpas} and \cite[Theorem 2.9.1]{aam:lpa}, the Leavitt path algebra $L_K(E)$ of an arbitrary graph $E$ is purely infinite simple if and only if $L_K(E)$ is simple and $E$ has a cycle. Since $G$ is finite, $S_{G, \mathfrak{r}}$ is a subgroup of $G$, and so, $\Gamma_{G, \mathfrak{r}}$ has a cycle by Proposition \ref{common} (6). These observations show the equivalence of (1) and (2), thus finishing the proof.
\end{proof}


\subsection{Stable rank}
Let $S$ be a unital ring containing an associate ring $R$ as a two-sided ideal. Following \cite{vas:tsroratdots}, a vector $(a_i)_{i=1}^n$  in $S$ is called $R$-\emph{unimodular} if $a_1-1; a_i \in R$ for $i>1$ and there exists $b_1-1; b_i \in R\, (i > 1)$ such that $\sum_{i=1}^n a_ib_i = 1$. We denote by $\text{sr}(R)$ the \emph{stable rank} of $R$, which is the least number $m$ for which for any $R$-unimodular vector	$(a_i)_{i=1}^m+1$ there exists $r_i \in R$ such that the vector $(a_i + r_ia_{m+1})_{i=1}^m$ is $R$-unimodular. If such	an $m$ does not exist, the stable rank of $R$ is defined to be infinite. 

In \cite[Theorem 2.8]{AP:srolpa} Ara and Pardo showed that the only possible values for the stable rank of the Leavitt path algebra of a row-finite graphs are $1$, $2$ and $\infty$. In \cite[Theorem 4.7]{lr:srolpafag} Larki and Riazi extended this to an arbitrary graph. In the following theorem, by using \cite[Theorem 4.7]{lr:srolpafag} and Theorems \ref{notconnected} and \ref{purely}, we compute the stable rank of Leavitt path algebras of Hopf graphs via ramification datas.

\begin{thm}\label{stablerank}
Let $G$ be a group with  a ramification data $\mathfrak{r}=\sum_{C\in \mathcal{C}}\mathfrak{r}_CC$ and $K$ an arbitrary field. Then
\begin{align*} \rm{sr}(L_K(\Gamma_{G, \mathfrak{r}}))= \begin{cases} 1&\textnormal{if }  S_{G,\mathfrak{r}} \textnormal{ is not a submonoid of } G\\ 
			\infty&\textnormal{if }  \sum_{C\in \mathcal{C}}\mathfrak{r}_C|C|\geq 2 \textnormal{ and } S_{G,\mathfrak{r}} \textnormal{ is a finite subgroup of } G\\
			2&\textnormal{otherwise}.\end{cases}\end{align*}
\end{thm}
\begin{proof} 	
If $S_{G,\mathfrak{r}}$ is not a submonoid of $G$, then  $\Gamma_{G, \mathfrak{r}}$ is acyclic, by Proposition \ref{common} (6). By  \cite[Theorem 4.7]{lr:srolpafag} (1), $\text{sr}(L_K(\Gamma_{G, \mathfrak{r}}))=1$. If $\sum_{C\in \mathcal{C}}\mathfrak{r}_C|C|\geq 2$ and $S_{G,\mathfrak{r}}$ is a finite subgroup of $G$, then by Theorem \ref{purely}, $L_K(\Delta_{G,\mathfrak{r}})$ is both a unital purely infinite simple ring and a quotient of $L_K(\Gamma_{G,\mathfrak{r}})$. By \cite[Theorem 4.7]{lr:srolpafag} (2), $\text{sr}(L_K(\Gamma_{G,\mathfrak{r}}))=\infty$. Otherwise, we have the following cases:
	
\emph{Case} 1: $S_{G,\mathfrak{r}}$ is a submonoid of $G$ and $\sum_{C\in \mathcal{C}}\mathfrak{r}_C|C|\leq 1$. We then have that $\Gamma_{G,\mathfrak{r}}$ has a cycle, by Proposition \ref{common} (6). By Corollary \ref{singlecycle}, $\Gamma_{G,\mathfrak{r}}$ is a disjoint union of single cycles, and so, $\text{sr}(L_K(\Gamma_{G,\mathfrak{r}}))=2$, by \cite[Theorem 4.7]{lr:srolpafag} (3).
	
\emph{Case} 2: $S_{G,\mathfrak{r}}$ is an infinite subgroup of $G$ and $\sum_{C\in \mathcal{C}}\mathfrak{r}_C|C|\geq 2$. Then, by Theorem \ref{purely}, $L_K(\Delta_{G,\mathfrak{r}})$ is non-unital purely infinite simple. This implies that $L_K(\Gamma_{G,\mathfrak{r}})$ has no unital purely infinite simple quotients, and so, $\text{sr}(L_K(\Gamma_{G,\mathfrak{r}}))=2$, by \cite[Theorem 4.7]{lr:srolpafag} (3).
	
\emph{Case} 3: $S_{G,\mathfrak{r}}$ is an infinite submonoid of $G$, but not a group  and $\sum_{C\in \mathcal{C}}\mathfrak{r}_C|C|\geq 2$. Then, since  $S_{G,\mathfrak{r}}$ is a submonoid of $G$ and by Proposition \ref{common} (6), $\Gamma_{G,\mathfrak{r}}$ contains a cycle. We next claim that $L_K(\Gamma_{G,\mathfrak{r}})$ has no unital purely infinite simple quotients. Indeed, by Theorem \ref{notconnected}, it is enough to show that $L_K(\Delta_{G,\mathfrak{r}})$ has no unital purely infinite simple quotients. Assume that $L_K(\Delta_{G,\mathfrak{r}})$ has a unital purely infinite simple quotient. By \cite[Proposition 3.3]{lr:srolpafag}, there exists a hereditary and saturated subset $H$ of $\Delta_{G,\mathfrak{r}}^0$ such that $L_K(\Delta_{G,\mathfrak{r}}\setminus H)$ is unital purely infinite simple, where $\Delta_{G,\mathfrak{r}}\setminus H$ is the quotient graph defined by:
\begin{center}
$(\Delta_{G,\mathfrak{r}}\setminus H)^0 = \Delta_{G,\mathfrak{r}}^0\setminus H$ and $(\Delta_{G,\mathfrak{r}}\setminus H)^1 = \{e\in \Delta_{G,\mathfrak{r}}^1\mid r(e)\notin H\}$.	
\end{center} 
In particular, we receive that $\Delta_{G,\mathfrak{r}}^0\setminus H$ is a nonempty finite set. Let $v\in \Delta_{G,\mathfrak{r}}^0\setminus H$. Since $S_{G,\mathfrak{r}}$ is not a group, there exists an element $c\in S_{G,\mathfrak{r}}$ such that $wc\neq 1_G$ for all $w\in S_{G,\mathfrak{r}}$. We then have that $\{vc^{-k}\}^{\infty}_{k=0}$ are distinct vertices in $\Delta_{G,\mathfrak{r}}$, where $c^0 := 1_G$. Since $vc^{-k} =(vc^{-k-1})c$, for each $k\ge 0$, there exists an edge $e_k$ in $\Delta_{G,\mathfrak{r}}$ such that 
$s(e_k) = vc^{-k-1}$ and $r(e_k) = vc^{-k}$, and so, there exists a path in $\Delta_{G,\mathfrak{r}}$ from $vc^{-k}$ to $v$ for all $k\ge 1$. Since $H$ is hereditary and $v\notin H$, $vc^{-k}\notin H$ for all $k$. This shows that $\Delta_{G,\mathfrak{r}}^0\setminus H$ is an infinite set, a contradiction, proving the claim. From these observations and \cite[Theorem 4.7 (3)]{lr:srolpafag}, we immediately obtain that  
$\text{sr}(L_K(\Gamma_{G,\mathfrak{r}}))=2$, thus finishing the proof.
\end{proof}

We end this subsection by presenting the following example.

\begin{exas} 
	
	(1)  Consider the Hopf graph $\Gamma_{G, \mathfrak{r}}$ of Example \ref{2.3} (1). We then have that $\sum_{C\in \mathcal{C}}\mathfrak{r}_C|C|=2$ and $S_{G, \mathfrak{r}}=\{id, (123), (132)\}$ is a finite subgroup of $G=S_3$, and so $\text{sr}(L_K(\Gamma_{G, \mathfrak{r}}))=\infty$, by Theorem \ref{stablerank}. 
	
	(2) Consider the Hopf graph $\Gamma_{G, \mathfrak{r}}$ of Example \ref{2.4} (1). We then have that $\sum_{C\in \mathcal{C}}\mathfrak{r}_C|C|=2$ and $S_{G, \mathfrak{r}}=\{0, 2n\mid n\geq 1\}$ is an infinite submonoid of $G=\mathbb{Z}$, and so  $\text{sr}(L_K(\Gamma_{G, \mathfrak{r}}))=2$, by Theorem \ref{stablerank}.
	
	We note that Leavitt path algebras of these graphs have infinite Gelfand-Kirillov dimension (Example \ref{exGK}).
	
	(3)  Consider the Hopf graph $\Gamma_{G, \mathfrak{r}}$ of Example \ref{2.4} (2). We then have that $S_{G, \mathfrak{r}}=\{2m+3n\mid m,n\geq 1\}$ is not a submonoid of $G=\mathbb{Z}$, and so $\text{sr}(L_K(\Gamma_{G, \mathfrak{r}}))=1$, by Theorem \ref{stablerank}.	
\end{exas}

%

\subsection{Finite dimensional representations}
We begin this subsection by recalling notions of graph theory introduced in \cite{ko:rolpa, ko:fdrolpa}. Let $E$ be a graph. For a cycle $c$ and a sink $v$ in $E$, we write $c\Rightarrow v$ if there exists a path in $E$ which starts in $c$ and ends at $v$. A sink $v$ in $E$ is called {\it maximal} if there is not any cycle $c$ in $E$ such that $c\Rightarrow v$. A cycle $c$ in $E$ is called {\it maximal} if there is not any cycle $d$ in $E$ which is different from a cyclic permutation of $c$ such that $d\Rightarrow c$. The {\it predecessors} of a vertex $v$ in $E$ is the set $E_{\ge v} := \{w\in E^0\mid w\ge v\}$ and the {\it predecessors} of a cycle $c$ in $E$ is the set $E_{\ge v}$, where $v$ is an arbitrary vertex on $c$.

In \cite[Theorem 6.5]{ko:rolpa} Ko\c{c} and \"{O}zaydin proved that the Leavitt path algebra $L_K(E)$ of a row-finite graphs $E$ with coefficients over a field $K$ has a nonzero finite dimensional representations if and only if $E$ has a maximal sink or cycle with finitely many predecessors. Moreover, they have classified all finite dimensional representations of Leavitt path algebras of row-finite graphs (see \cite[Theorem 4.7]{ko:fdrolpa}). In the following theorem, based essentially on these results, we classify all finite dimensional representations of Leavitt path algebras of Hopf graphs via ramification datas.

\begin{thm}\label{fdm}
Let $G$ be a group with  a finite support ramification data $\mathfrak{r}=\sum_{C\in \mathcal{C}}\mathfrak{r}_CC$ and $K$ an arbitrary field. Then $L_K(\Gamma_{G, \mathfrak{r}})$ has a nonzero finite dimensional representation if and only if one of the following conditions holds:
	
$(1)$ $\mathfrak{r}=0$;
	
$(2)$ $\sum_{C\in \mathcal{C}}\mathfrak{r}_C|C|=1$ and $S_{G, \mathfrak{r}}$ is a submonoid of $G$.
	
Furthermore, if the above conditions are satisfied, then 
\begin{align*}
\mathcal{M}_{L_K(\Gamma_{G, \mathfrak{r}})}^{fd}\backsimeq \begin{cases}
(\mathcal{M}_K^{fd})^{(G)} & \textnormal{ if } \mathfrak{r}=0\\
(\mathcal{M}_{K[x,x^{-1}]}^{fd})^{(G/S_{G, \mathfrak{r}})} & \textnormal{ otherwise,}
\end{cases}
\end{align*} where $\mathcal{M}_A^{fd}$ is the category of finite dimensional $A$-modules, $\mathcal{N}^{(X)}$ is the $X$-indexed direct sum of copies of the category $\mathcal{N}$, and $\backsimeq$ denotes equivalence of categories.
\end{thm}
\begin{proof} 
$(\Longrightarrow)$. 	Assume that $L_K(\Gamma_{G, \mathfrak{r}})$ has a nonzero finite dimensional representation. By \cite[Theorem 6.5]{ko:rolpa}, $\Gamma_{G, \mathfrak{r}}$ has a maximal sink or cycle with finitely many predecessors. We consider the following cases:

\emph{Case }1: $\sum_{C\in \mathcal{C}}\mathfrak{r}_C|C|\geq 2$. Then,  by Proposition \ref{common} (5) and Lemma \ref{sequencecycle} respectively, $\Gamma_{G, \mathfrak{r}}$ has neither maximal sinks nor maximal cycles, a contradiction.	

\emph{Case }2: $\sum_{C\in \mathcal{C}}\mathfrak{r}_C|C|=1$ and $S_{G, \mathfrak{r}}$ is not a submonoid of $G$. Then,  by Propositions \ref{common} (5) and (6)  respectively, $\Gamma_{G, \mathfrak{r}}$ has neither sinks nor cycles. So, $\Gamma_{G, \mathfrak{r}}$ has neither maximal sinks nor maximal cycles, a contradiction.	

Therefore, in any case, we arrive at a contradiction, thus proving the statement.

$(\Longleftarrow)$. If $\mathfrak{r}=0$, we then have $L_K(\Gamma_{G, \mathfrak{r}}) = K^{(G)}$, by Corollary \ref{singlecycle} (1). This implies that $\mathcal{M}_{L_K(\Gamma_{G, \mathfrak{r}})}^{fd} \backsimeq (\mathcal{M}_K^{fd})^{(G)}$. Consider the case when $\sum_{C\in \mathcal{C}}\mathfrak{r}_C|C|= 1$ and $S_{G, \mathfrak{r}}$ is a submonoid of $G$. By Corollary \ref{singlecycle} (2), we immediately obtain that $S_{G, \mathfrak{r}}$ is a finite subgroup of $G$ and $L_K(\Gamma_{G, \mathfrak{r}}) =(M_{|S_{G, \mathfrak{r}}|}(K[x, x^{-1}]))^{(G/S_{G, \mathfrak{r}})}$, and so, $\mathcal{M}_{L_K(\Gamma_{G, \mathfrak{r}})}^{fd} \backsimeq (\mathcal{M}_{K[x, x^{-1}]}^{fd})^{(G/S_{G, \mathfrak{r}})}$, thus finishing the proof.	
\end{proof}

A unital ring $R$ is said to have \emph{Invariant Basis Number} if,
for any pair of positive integers $m \text{ and } n$, $R^m\cong R^n$ (as right modules) implies that $m=n$. A unital ring $R$ is said to have \emph{Unbounded Generating Number} if for all $m, n \in \mathbb{N}$ and any right $R$-module $K$, $R^n \cong R^m \oplus K$ (as right $R$-modules) implies that $n \geq m$. A straightforward computation immediately establishes that if $R$ has Unbounded Generating Number, then $R$ has Invariant Basis Number.

Criteria for Leavitt path algebras of finite graphs to have Invariant Basis Number have been established in  \cite[Theorem 3.1]{ko:clpaatibnp} and \cite[Theorem 3.5]{np:tsolpaatibnp}. 
In \cite[Theorem 3.16]{anp:lpahugn} Abrams and the authors completely classified finite graphs $E$ for which the Leavitt path algebra $L_K(E)$ of $E$ with coefficients over a field $K$ have Unbounded Generating Number. In \cite[Example 3.19]{anp:lpahugn} Abrams and the authors established that, within the class of Leavitt path algebras, the Invariant Basis Number property is strictly weaker than the Unbounded Generating Number property. 
However, these properties are equivalent to each other within the class of Leavitt path algebras of Cayley graphs (\cite[Corollary 4.3]{np:tsolpaatibnp}). In the following corollary (Corollary \ref{IBNUGN}), we prove that these properties are equivalent to each other for Leavitt path algebras of Hopf graphs. Before doing so, we need some useful notions and facts.

Following \cite{amp:nktfga}, for any directed graph
$E=(E^0, E^1, s, r)$ we define the monoid $M_E$ as follows. We denote by $T$ the free abelian
monoid (written additively) with generators $E^0$ and define relations on $T$ by setting
\begin{center}
	$v = \sum_{e\in s^{-1}(v)}r(e)$
\end{center}
for every regular vertex $v\in E^0$.
Let $\sim_{E}$ be the congruence relation on $T$ generated by these relations.
Then $M_E = T/_{\sim_E}$, and we also denote an element of $M_E$ by $[x]$,
where $x\in T$. 

%
%

Following  \cite[Corollary 3.4]{np:tsolpaatibnp}, the Leavitt path algebra $L_K(E)$ of a finite graph $E$ with coefficients in a field $K$ has Invariant Basis Number if and only if for any pair of positive integers $m$ and $n$,
	\begin{center}
		if $\ m[\sum_{v\in E^0}v] = n[\sum_{v\in E^0}v]$ in $M_E$, then $m = n$.
	\end{center}
\medskip

We end this article with the following fact which extends \cite[Theorem 4.2]{np:tsolpaatibnp} to Hopf graphs.

\begin{cor}\label{IBNUGN}
Let $G$ be a finite group with  a ramification data $\mathfrak{r}=\sum_{C\in \mathcal{C}}\mathfrak{r}_CC$ and $K$ an arbitrary field. Then, the following statements are equivalent:
	
$(1)$ $L_K(\Gamma_{G, \mathfrak{r}})$ has a nonzero finite dimensional module;
	
$(2)$ $L_K(\Gamma_{G, \mathfrak{r}})$ has Unbounded Generating Number;
	
$(3)$ $L_K(\Gamma_{G, \mathfrak{r}})$ has Invariant Basis Number;
	
$(4)$ $\sum_{C\in \mathcal{C}}\mathfrak{r}_C|C|\leq 1$.
\end{cor}
\begin{proof}
(1) $\Longleftrightarrow$ (2). It follows from \cite[Theorem 3.16]{anp:lpahugn} and \cite[Theorem 6.5]{ko:rolpa}.
	
(2) $\Longrightarrow$ (3). It is obvious.
	
(3) $\Longrightarrow$ (4). Assume that $L_K(\Gamma_{G, \mathfrak{r}})$ has Invariant Basis Number, and $m:=\sum_{C\in \mathcal{C}}\mathfrak{r}_C|C|\geq 2$. Write 
\begin{center}
$\text{supp}(\mathfrak{r}) = \{C_1, C_2, \cdots, C_k\}$ and $C_i = \{c_{i1}, c_{i2}, \cdots, c_{in_i}\}$.		
\end{center}
We then have 
$G = Gc_{ij}$ for all  $1\le i\le k$ and $1\le j\le n_i$ (since $G$ is a finite group) and every vertex $g$ in $\Gamma_{G, \mathfrak{r}}$ emits exactly $\mathfrak{r}_{C_i}$ edges to the vertex $gc_{ij}$ for all $1\le i\le k$ and $1\le j \le n_i$, and $g$ emits only to these vertices $gc_{ij}$. This implies that 
\begin{center}
$[g] = \sum^{n_1}_{j=1}\mathfrak{r}_{C_1}[gc_{1j}] + \sum^{n_2}_{j=1}\mathfrak{r}_{C_2}[gc_{2j}]  + \cdots + \sum^{n_k}_{j=1}\mathfrak{r}_{C_k}[gc_{kj}]$ in $M_{\Gamma_{G, \mathfrak{r}}}$
\end{center}
and 
\begin{center}
	$[\sum_{g\in G}g] =  \sum^{n_1}_{j=1}\mathfrak{r}_{C_1}[\sum_{g\in G}gc_{1j}]  + \cdots + \sum^{n_k}_{j=1}\mathfrak{r}_{C_k}[\sum_{g\in G}gc_{kj}]$ in $M_{\Gamma_{G, \mathfrak{r}}}.$
\end{center}
Since $G = Gc_{ij}$ for all $1\le i\le k$ and $1\le j\le n_i$, we obtain that 
\begin{center}
$\sum_{g\in G}g = \sum_{g\in G}gc_{ij}$ for all $1\le i\le k$ and $1\le j\le n_i,$ 
\end{center} 
so 

\begin{center}
$[\sum_{g\in G}g] = \sum^k_{i=1}n_i\mathfrak{r}_{C_i}[\sum_{g\in G}g] = m[\sum_{g\in G}g]$ in $M_{\Gamma_{G, \mathfrak{r}}},$
\end{center}
showing that  $L_K(\Gamma_{G, \mathfrak{r}})$ has no Invariant Basis Number by  \cite[Corollary 3.4]{np:tsolpaatibnp}, a contradiction. Therefore, we must have $\sum_{C\in \mathcal{C}}\mathfrak{r}_C|C|\leq 1$.
	
(4) $\Longrightarrow$ (1). It follows from Theorem \ref{fdm} and the fact that if $G$ is a finite group, then  $S_{G, \mathfrak{r}}$ is a subgroup of $G$, thus finishing the proof.
\end{proof}

\vskip 0.5 cm \vskip 0.5cm {

\end{document}